\documentclass[12pt]{amsart}
\usepackage{researchmacros,amsthm}
\usepackage{fullpage}
\title{Orbit harmonics for the union of two orbits}
\author{Sean T. Griffin}
\email{stgriffin@ucdavis.edu}
\address{Department of Mathematics, University of California, Davis, One Shields Avenue, Davis, CA 95616}
\date{\today}

\newtheorem{theorem}{Theorem}

\newtheorem{lemma}{Lemma}
\theoremstyle{definition}
\newtheorem{example}{Example}

\renewcommand{\top}{\mathrm{top}}
\newcommand{\Frob}{\mathrm{Frob}}

\begin{document}
\maketitle

\begin{abstract}
Garsia and Procesi, in their study of Springer's representation, proved that the cohomology ring of a Springer fiber is isomorphic to the associated graded ring of the coordinate ring of the $S_n$ orbit of a single point in $\mathbb{C}^n$. This construction was an essential tool in their analysis of the Springer representation, and variations of it have reappeared recently in several other combinatorial and geometric contexts under the name \emph{orbit harmonics}. In this article, we analyze the orbit harmonics of a union of two $S_n$ orbits. We prove that when the coordinate sums of the two orbits are different, the corresponding graded $S_n$ representation is a direct sum of two Springer representations, one of which is shifted in degree by 1.
\end{abstract}

\section{Introduction}

Let $Y\subset \bC^n$ be a finite set of points. Since $Y$ is finite, it is a variety in $\bC^n$ whose defining ideal is the intersection of all the maximal ideals corresponding to each point in $Y$, so $I(Y) = \cap_{y\in Y}(x_1-y_1,x_2-y_2,\dots,x_n-y_n)$ and whose coordinate ring is $A(Y) = \bC[x_1,\dots,x_n]/I(Y)$.

Given a polynomial $f$ in the variables $\{x_1,\dots, x_n\}$, suppose $f = f_0+f_1+\dots+f_d$, where $f_i$ is the $i$th degree component of $f$ and such that $f_d\neq 0$. Let $\top(f) = f_d$, the top degree component of $f$. To the set $Y$ we associate the \emph{top degree ideal} or \emph{associated graded ideal} \[T(Y) = (\top(f) : f\in I(Y)).\] Note that $T(Y)$ is a homogeneous ideal (but will in general not be a monomial idea). We write \[\widehat A(Y) := \bC[x_1,\dots,x_n]/T(Y),\]
which is isomorphic to the associated graded ring of $A(Y)$.

Our interest is primarily in the case when the set $Y$ is closed under the action of the symmetric group $S_n$ permuting coordinates. In this case, $I(Y)$ is an $S_n$-invariant ideal, hence $T(Y)$ is also $S_n$-invariant. Therefore, $\widehat A(Y)$ is a graded $S_n$-module whose dimension over $\bC$ is $|Y|$. Given a finite set $X\subset \bC^n$, let $S_n\langle X\rangle$ be the union of the orbits of the points in $X$ under the $S_n$ action. 

The construction above was used by Garsia and Procesi~\cite{Garsia-Procesi} in their study of the graded $S_n$-module type of the \emph{Springer fiber} $\mathcal{B}_\lambda$ associated to a nilpotent matrix with Jordan type $\lambda$. Their starting point was a presentation of the cohomology ring of a Springer fiber,  
\[H^*(\mathcal{B}_\lambda;\bC) \cong \bC[x_1,\dots, x_n]/I_\lambda,\]
where $I_\lambda$ is the \emph{Tanisaki ideal}. The following theorem was a key step in their analysis.

\begin{theorem}[\cite{Garsia-Procesi}]\label{thm:GPthm}
  Let $P\in \bC^n$ and $Y = S_n\langle P\rangle$, the orbit of a single point. Let $\lambda\vdash n$ be the partition recording the multiplicities of the coordinates of the point P. Then $T(Y)= I_{\lambda}$.
  \end{theorem}

Given a finite-dimensional $S_n$-module $V$, let $\Frob(V)$ be the \emph{Frobenius characteristic} of $V$, which is a symmetric function in an infinite set of variables $\{x_1,x_2,\dots\}$. Given a graded finite-dimensional $S_n$-module $V = \bigoplus_{i\geq 0}V_i$, let $\Frob_q(V) = \sum_{i\geq 0} q^i\Frob(V_i)$ be the \emph{graded Frobenius characteristic} of $V$, which is a symmetric function in variables $\{x_1,x_2,\dots\}$ with coefficients in $\bZ_{\geq 0}[q]$.
Let $s_\lambda(x)$ be the \emph{Schur function} associated to the partition $\lambda$, and let $\widetilde{H}_\lambda (x;q)$ be the \emph{modified Hall-Littlewood function}. Let $\leq$ be the usual \emph{dominance order} on partitions defined as follows: Given two partitions $\lambda$ and $\mu$ of $n$, we write $\lambda\leq \mu$ if and only if for all $i$, $\lambda_1+\cdots+ \lambda_i\leq \mu_1 + \cdots \mu_i$. See~\cite{Macdonald} for more details on symmetric functions and the Frobenius characteristic. 

Garsia and Procesi used Theorem~\ref{thm:GPthm} to prove the following.

\begin{theorem}[\cite{Garsia-Procesi}]
  Suppose $Y$ is as in the previous theorem. Then 
  \[
    \Frob_q(R_\lambda) = \Frob_q(\widehat A(Y)) =\widetilde H_{\lambda}(x;q).
  \]
\end{theorem}

The associated graded rings $\widehat A(Y)$ for various sets of points $Y$ has recently appeared in several other contexts,  including spanning configuration spaces, cyclic sieving phenomena, and graded $S_n$-modules that generalize $R_\lambda$, see ~\cite{Chan-Rhoades, GriffinOSP, HRS1,OhRhoades, Brendon-SpanningSub, Rhoades-Wilson-Stirling}. See also the excellent survey of related work~\cite{Rhoades-book}, in which the study of orbit harmonics for general $S_n$-orbits $Y$ was proposed.

In this article, we consider the problem of computing the graded Frobenius characteristic of $\widehat A(Y)$ when $Y$ is the union of two distinct orbits, $Y = S_n\langle P,Q\rangle$. We state our main theorem, Theorem~\ref{thm:MainTheorem}, in the next section, which gives an expansion of the graded Frobenius characteristic of $\widehat A(Y)$ into modified Hall-Littlewood functions when the coordinate sums of $P$ and $Q$ are different. We show in Example~\ref{ex:NotPositive} that the graded Frobenius characteristic of the ring $\widehat A(Y)$ does not necessarily have a positive expansion into modified Hall-Littlewoods when the coordinate sums of $P$ and $Q$ are equal.

\section{Orbit harmonics for the union of two orbits}

\begin{lemma}\label{lemma:elementReps}
If $P\in \bC^n$, and $\widetilde f\in T(S_n\langle P\rangle)$ is a homogeneous polynomial of degree $d$, then $\widetilde f = \top(f)$ for some $f\in I(S_n\langle P\rangle)$ of degree $d$.
\end{lemma}

\begin{proof}
Without loss of generality, let us assume $\widetilde f\neq 0$. By definition of $T(S_n\langle P\rangle)$,
  \[
    \widetilde f = \sum_i g^{(i)}\top(h^{(i)}),
  \]
  for some $h^{(i)}\in I(S_n\langle P\rangle)$ and $g^{(i)}\in \bC[x_1,\dots, x_n]$. Letting $m_i = \deg(h^{(i)})$, then
  \[
    \widetilde f = \sum_i (g^{(i)})_{d-m_i} \top(h^{(i)}).
  \]
  Let $I$ be the subset of indices such that $(g^{(i)})_{d-m_i}\neq 0$ (where the subscript means take the $d-m_i$ degree homogeneous component of $g^{(i)}$), then
  \[
    \widetilde f = \sum_{i\in I}(g^{(i)})_{d-m_i}\top(h^{(i)}).
  \]
  Since $\widetilde f \neq 0$, the $d$th degree component of $\sum_i (g^{(i)})_{d-m_i}h^{(i)}$ is nonzero, thus $\widetilde f$ is the top degree component of this sum,
  \[
    \widetilde f = \top\left(\sum_{i\in I} (g^{(i)})_{d-m_i}h^{(i)}\right).
  \]
  Since $\sum_{i\in I} (g^{(i)})_{d-m_i}h^{(i)}\in I(S_n\langle P\rangle)$, we are done.
\end{proof}

\begin{theorem}\label{thm:MainTheorem}
  Let $P,Q\in \bC^n$ such that the multiplicites of the coordinates of $P$ and $Q$ are $\lambda$ and $\mu$, respectively, and let $Y = S_n\langle P,Q\rangle$. Let $\Sigma(P) = \sum_i P_i$ and $\Sigma(Q) = \sum_i Q_i$. If $\lambda\leq \mu$ and $\Sigma(P)\neq \Sigma(Q)$, then
  \[
    \widehat A(Y) \cong_{S_n} R_\lambda \oplus R_\mu(-1)
  \]
  as graded $S_n$-modules, where here the (-1) means that $R_\mu$ is shifted up in degree by $1$. Equivalently, we have
  \[
    \Frob_q(\widehat A(Y)) = \widetilde H_{\lambda} + q\widetilde H_{\mu}.
  \]
\end{theorem}

Before we prove the theorem, we need two lemmas relating the various ideals.
\begin{lemma}\label{lemma:idealProduct}
  Given $P,Q\in \bC^n$ and $Y = S_n\langle P,Q\rangle$,
  \[
    T(S_n\langle Q\rangle)(x_1+\dots+x_n)\subseteq T(Y).
  \]
\end{lemma}
\begin{proof}
Given any generator $\top(f)\in T(S_n\langle Q\rangle)$ for $f\in I(S_n\langle Q\rangle)$, the function 
\[(x_1+\dots+x_n-\Sigma(P))f\] vanishes on $S_n\langle P,Q\rangle$, so
  \[
    \top(f)\cdot (x_1+\dots+x_n) = \top((x_1+\dots+x_n-\Sigma(P))f) \in T(Y).
  \]
  Hence, we have the desired containment.
\end{proof}

\begin{lemma}\label{lemma:idealEquation}
  Given $P$, $Q$, and $Y$ as above, then
  \[
    T(S_n\langle P\rangle) = T(Y) + (x_1+\dots+x_n).
  \]
\end{lemma}

\begin{proof}
  As in the Introduction, let $I_\lambda= T(S_n\langle P\rangle)$ and $I_\mu = T(S_n\langle Q\rangle)$, which are the Tanisaki ideals corresponding to $\lambda$ and $\mu$, respectively. By~\cite{Garsia-Procesi}, if $\lambda\leq \mu$, then $I_\lambda\subseteq I_\mu$.

  First, we prove that the containment ($\supseteq$) in the statement of the lemma holds. Since $S_n\langle P\rangle\subseteq S_n\langle P,Q\rangle$, we have that $I(S_n\langle P\rangle)\supseteq I(S_n\langle P,Q\rangle)$. Hence, we also have $T(S_n\langle P\rangle)\supseteq T(S_n\langle P,Q\rangle)$. Since the function $x_1+\dots+x_n-\Sigma(P)$ vanishes on $P$, we have that
  \begin{equation}\label{eq:idealContainment}
    T(S_n\langle P\rangle)\supseteq (x_1+\dots+x_n).
  \end{equation}
  Hence we have
  \[
    T(S_n\langle P\rangle)\supseteq T(S_n\langle P,Q\rangle) + (x_1+\dots+x_n).
  \]

  To prove the other containment, let $\top(f)\in T(S_n\langle P\rangle)$, where $f\in I(S_n\langle P\rangle)$. It suffices to show that it can be written as the sum of an element of $T(S_n\langle P,Q\rangle)$ and a multiple of $x_1+\dots+x_n$. If $\top(f)\in (x_1+\dots+x_n)$, then we are done, so suppose otherwise.

  By the containment $I_\lambda\subseteq I_\mu$, we have $\top(f)\in I_\mu$. Since $\top(f)$ is homogeneous, we have that $\widetilde f = \top(g)$ for some $g\in I(S_n\langle Q\rangle)$ by Lemma~\ref{lemma:elementReps}. Let
  \[
    h = f\left(\sum_i x_i - \Sigma(Q)\right) - g\left(\sum_i x_i-\Sigma(P)\right).
  \]
  Since $f$ vanishes on $S_n \langle P\rangle$ and $\sum_i x_i-\Sigma(Q)$ vanishes on $S_n \langle Q\rangle$, their product vanishes on $S_n\langle P,Q\rangle$, and similarly for the second term in $h$. Hence, $h \in I(S_n\langle P,Q\rangle)$ so $\top(h)\in T(S_n\langle P,Q\rangle)$.

  Since $\top(f) = \top(g) =\widetilde f$, the $(d+1)$st degree component of $h$ vanishes. The $d$th degree component of $h$ is
  \begin{align}\label{eq:toph}
    (f_{d-1}-g_{d-1})\left(\sum_i x_i\right) + \widetilde f(-\Sigma(Q)+\Sigma(P)).
  \end{align}
  If this vanishes, then we would have $\widetilde f \in (x_1+\dots+x_n)$, contradicting our earlier assumption. Therefore, the $d$th degree component of $h$ must not vanish and $\top(h)$ must equal~\eqref{eq:toph}. Since $\Sigma(P)\neq \Sigma(Q)$, we can solve for $\widetilde f$ to get
  \[
    \widetilde f = \frac{1}{\Sigma(P)-\Sigma(Q)} (\top(h) - (f_{d-1}-g_{d-1})(x_1+\dots+x_n))
  \]
which is clearly an element of $T(S_n\langle P,Q\rangle) + (x_1+\dots + x_n)$, and we are done.
\end{proof}

Now we will use the lemmas to prove the theorem.

\begin{proof}[Proof of Theorem~\ref{thm:MainTheorem}]
  By Lemma~\ref{lemma:idealProduct}, we have a well-defined homomorphism
  \[
    \frac{\bC[x_1,\dots,x_n]}{T(S_n\langle Q\rangle)}\to \frac{\bC[x_1,\dots,x_n]}{T(Y)}
  \]
  induced from multiplication by $x_1+\dots+x_n$. Furthermore, we have an exact sequence
  \[
    \frac{\bC[x_1,\dots,x_n]}{T(S_n\langle Q\rangle)}\to \frac{\bC[x_1,\dots,x_n]}{T(Y)}\to \frac{\bC[x_1,\dots,x_n]}{T(S_n\langle P\rangle)}\to 0,
  \]
  where the second map is a quotient map induced by the ideal containment \eqref{eq:idealContainment}. Indeed, the sequence is exact at $\bC[x_1,\dots,x_n]/T(Y)$ by Lemma~\ref{lemma:idealEquation}. But since $\dim_\bC(R_\mu) + \dim_\bC(R_\lambda) = |Y| = \dim_\bC(\widehat A(Y))$, the first map must also be an injection, so we have an exact sequence of graded $\bC[x_1,\dots,x_n]$-modules,
  \[
    0\to R_\mu(-1)\to \widehat A(Y)\to R_\lambda\to 0.
  \]
  But this is also clearly an exact sequence of graded $S_n$-modules. Over $\bC$, any exact sequence of graded $S_n$-modules splits, so we have the desired isomorphism.
\end{proof}

\begin{example}\label{ex:NotPositive}
In the case where $\Sigma(P) = \Sigma(Q)$, the graded Frobenius characteristic of $\widehat A(Y)$ does not necessarily expand positively in the modified Hall-Littlewood function basis. For example, when $P = (0,0,4)$, $Q = (1,1,2)$, and $Y = S_3\langle P,Q\rangle$, we have computed in Sage~\cite{sage} that
\[
\Frob_q(\widehat A(Y)) = t^2(s_{2,1} + s_3) + ts_{2,1} + s_3 = (t+1)\widetilde H_{2,1} + (t^2-t) \widetilde H_{3},
\]
which contains the term $-t\widetilde H_3$.
\end{example}

\section{Acknowledgements}

Thanks to Jake Levinson for several helpful conversations pertaining to this work.

\bibliographystyle{hsiam}
\bibliography{Springer.bib}

\end{document}